\documentclass[11pt, reqno]{amsart}
\usepackage[pagewise, displaymath, mathlines]{lineno}
\usepackage{subfig, enumerate, tikz-cd}
\usepackage{bookmark}
\usepackage{multirow}
\usepackage{filecontents}
\usepackage[shortlabels]{enumitem}
\usepackage[section] {placeins}
\usepackage{cases}
\usepackage{amscd,amsfonts,amssymb,amsmath,tikz-cd}
\usepackage{graphicx}
\usepackage{bm}
\usepackage{hyperref}
\usepackage{epsfig}
\usepackage{mathtools}
\usepackage{float}
\usepackage{color}
\usepackage[normalem]{ulem}
\usepackage[sort,nocompress]{cite}
\usepackage{array}
\newtheorem{theorem}{Theorem}[section]

\newtheorem{proposition}[theorem]{Proposition}

\newtheorem{lemma}[theorem]{Lemma}
\newtheorem{remark}[theorem]{Remark}

\theoremstyle{definition}

\newtheorem{example}{Example}

\newtheorem*{ack}{Acknowledgement}
\newcommand{\R}{\operatorname{R}}

\newcommand{\I}{\operatorname{I}}

\newcommand{\M}{\operatorname{M}}
\newcommand{\overbar}[1]{\mkern 1.5mu\overline{\mkern-1.5mu#1\mkern-1.5mu}\mkern 1.5mu}

\usepackage[autostyle]{csquotes}
\usepackage{caption}
\usepackage{wrapfig}
\usepackage{calc}

\makeatletter
\renewcommand\subsection{\@startsection{subsection}{2}%
  \z@{.5\linespacing\@plus.7\linespacing}{.5em}%
  {\normalfont\scshape\bfseries}}

\renewcommand{\@secnumfont}{\bfseries} 
\makeatother

\begin{document}


\title{Quasitoric representation of generalized braids}
\author{Neha Nanda, Manpreet Singh}

\address{LMNO, Université de Caen Normandie}
\email{nehananda94@gmail.com}

\address{University of South Florida, Tampa, FL, 33620}
\email{manpreet.math23@gmail.com}


\maketitle

\begin{abstract}
In this paper, we define generalized braid theories in alignment with the language of Fenn and Bartholomew \cite{MR4474096}  for knot theories, and compute a generating set for the pure generalized braid theories. Using this, we prove that every oriented normal generalized knot is the closure of a quasitoric normal generalized braid. Further, we prove that the set of quasitoric normal generalized braids forms a subgroup of normal generalized braid group.
\end{abstract}

\section{Introduction}
Classical knot theory studies knots as embedded, disjoint circles in Euclidean $3$-space, considered up to isotopy. The Alexander and Markov theorems are foundational results in this field, stating that every oriented link can be represented as the closure of a braid. Additionally, they establish that if a link is represented by two different braids (with possibly varying numbers of strands), the two braids are related by a finite sequence of {\em Markov moves}. Over the last few decades, many generalizations of knots and braid groups have been introduced, each of them an interest of its own. Fenn \cite{MR3381323} developed a meta-theory of knot theories which is termed as {\em $($normal$)$ generalized knot theories}, while Bartholomew-Fenn \cite{MR4474096} explored which of the Alexander-Markov theorems can be extended to these generalized theories. In continuation, we define generalized braid theories and compute the generating sets for the pure subgroups of generalized braid groups. These results align with the generating sets found in existing literature on generalized braid groups, such as virtual braid groups \cite{MR2128039}, virtual twin groups \cite{MR4651964}, unrestricted braid groups \cite{MR3415242}, welded braid groups \cite{MR873421,MR600411}, universal braid groups, and (extended) singular braid groups \cite{2022arXiv221208267B}.\\

Lamm \cite{MR1941420,2012arXiv1210.4543L} and Manturov \cite{MR1881552} independently proved that it suffices to consider a certain subclass of classical braids to encode isotopic links in the $3$-space via the Alexander theorem. Specifically, the notion of quasitoric braid generalizes toric braids which are braids whose closures form torus links in three-dimensional space, and it is proved that every oriented link is the closure of a quasitoric braid, and that the set of quasitoric braids with a fixed number of strands forms a subgroup of the braid group.The notion of a quasitoric representation of classical braids has been applied, for example, in the computation of the Casson invariant for integral homology 3-spheres obtained by performing Dehn surgery on specific knots and links \cite{MR2589233}. Additionally, knot invariants, homologies, and volume bounds have been explored in the context of weaving knots using this perspective \cite{MR4272646, MR3584259}. Recently, it was shown that any link can be represented by a diagram, viewed as a 4-valent graph embedded in the 2-sphere, where the faces of the graph are limited to triangles and quadrilaterals \cite{2023arXiv230814118S}. Similar to the braid index for the links, the quasitoric braid index of a link was introduced in \cite{MR3418559}, one application of this index is the result that the unknotting number of the knot $10_{139}$ is equal to $4$.\\

In this work, we define the notion of a quasitoric generalized braid, and prove that every generalized knot can be represented as the closure of a quasitoric generalized braid. In particular, the result holds for existing knot theories like classical knots, welded knots, free knots, singular knots, universal knots, virtual doodles, and recover the results for the case of virtual knots \cite{MR3335879}.  More recently, Genki \cite{doi:10.1142/S0218216524500469} computed the minimal generating sets and abelianization of the quasitoric braid group.\\

The goal of this paper is to provide a unified framework for the various braid theories and explore them in a general context. Specifically, we begin in Section \ref{Generalized-Braids-Section} by defining (regular and normal) generalized braid theory. In Section \ref{Pure-Generators-Section}, we compute a generating set for the pure regular generalized braid group (Theorem \ref{gPB_n-generators}). In Section \ref{Quasitoric-Braid-Section}, we introduce quasitoric generalized braids and prove that every oriented normal generalized link is the closure of a quasitoric normal generalized braid (Theorem \ref{Quasitoric-Main-Theorem}). Finally, we show that the set of quasitoric normal generalized braids forms a subgroup of the normal generalized braid group (Theorem \ref{Subgroup-Theorem}).

\section{Generalized braids}\label{Generalized-Braids-Section}
Consider $Q_n$ be a set of $n$ points in $\mathbb{R}$. A {\it generalized braid diagram} on $n$ strands is a subset $D$ of $\mathbb{R} \times [0,1]$ consisting of intervals called strands with boundary $\partial D = Q_n \times \{ 0 ,1\}$ satisfying:
\begin{enumerate}
\item[(i)] Each strand is monotonic.
\item[(ii)] The set of all crossings of the diagram $D$ consists of finitely many transverse double points of $D$ labeled by a tag indicated by a roman letter say `$a$'. This tag determines the \textit{crossing type} of the double point and how it behaves under the Reidemeister moves. Some tags come with \textit{glyph} (some decoration) such as the well-known breaking of arcs depicted under and over crossings.
\item[(iii)] Each tag has a positive version `$a$' and a negative version `$\bar{a}$', which may or may not be different. 
\item[(iv)] The diagrams are considered up to isotopy of the plane fixing the end points of the strands.
\end{enumerate} 

The {\it Reidemeister moves} or {\it $R$-moves} take one braid diagram to another in any of the ways shown in Figure \ref{R-Moves}. A {\it generalized braid theory} will define which of these moves are allowed and which are not. A {\it generalized braid} is an equivalence class of all braid diagrams related by a finite sequence of R-moves allowed in the corresponding generalized braid theory. We assume that there are finitely many tags in a generalized braid theory, and from now on we drop the term ``generalized" unless specified otherwise. 

\begin{figure}[tph]
    \centering
    \includegraphics[width=0.6\linewidth]{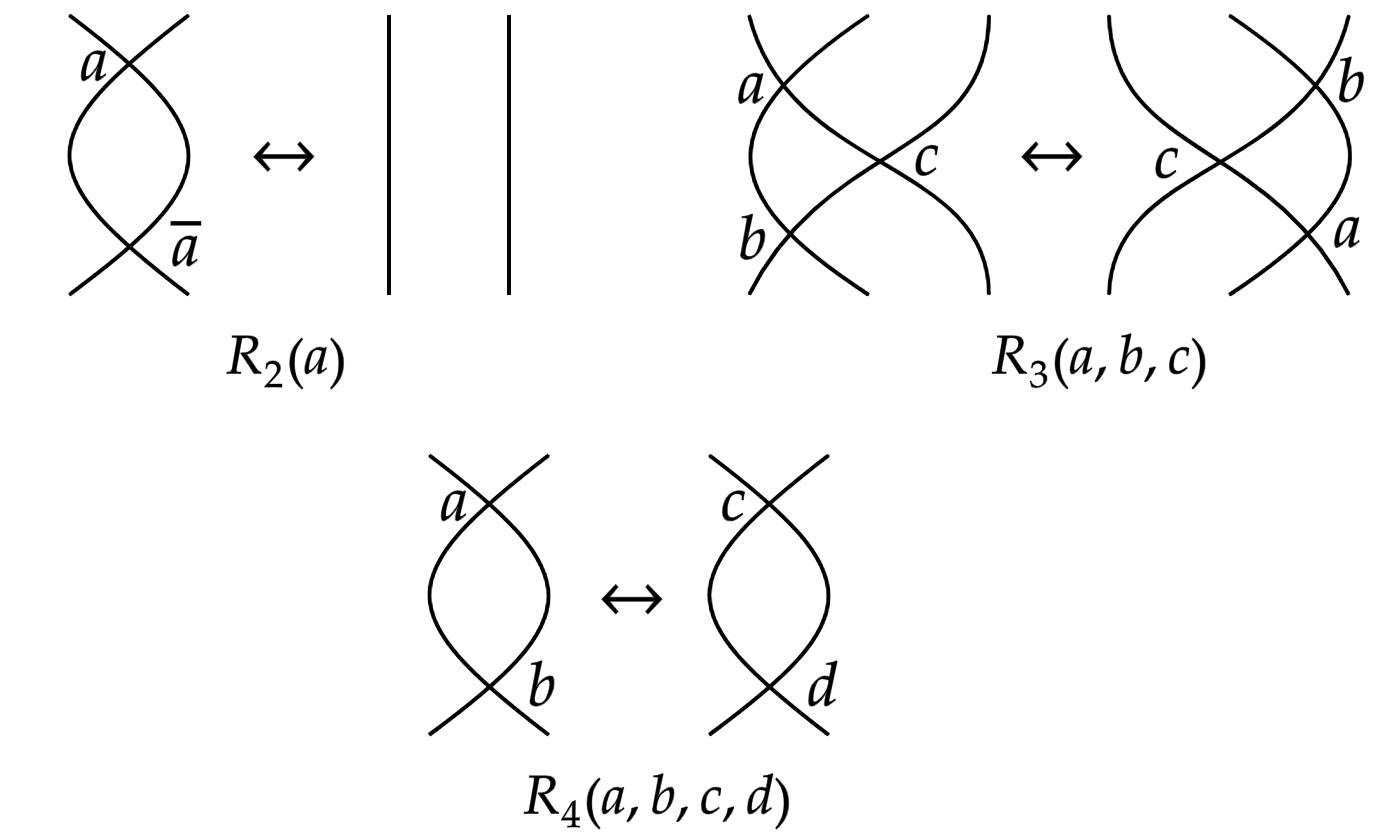}
    \caption{R-moves}
    \label{R-Moves}
\end{figure}

A braid theory is said to be {\it regular} if $\R_2(a)$ moves are allowed for all tags `$a$', and the braid is called {\it regular}.

\begin{proposition}
In a regular braid theory, the set $gB_n$ of all braids with $n$ strands forms a group under the operation of concatenation.
\end{proposition}

\begin{proof}
We begin by noting that the braid represented by a diagram of $n$ strands with no crossings is the identity element of the set $gB_n$ of regular braids. For each $i =1, 2, \dots , n-1$ and each tag $a$, let us define $a_i$ and $a_i^{-1}$ to be the regular braid represented by diagrams as in Figure \ref{Elementary-Braid}. For any arbitrary element $\beta$ in $gB_n$, it is easy to notice that $\beta$ can be expressed as composition of finitely many elementary braids with different tags. Since the braid theory is regular, for each tag $a$, the move $R_2(a)$ is allowed which implies that $a_i^{-1}$ is the inverse for $a_i$ for all $i$. As a result, the inverse $\beta^{-1}$ of each regular braid $\beta$ exists. 
\end{proof}

\begin{figure}[tph]
    \centering
    \includegraphics[width=0.8\linewidth]{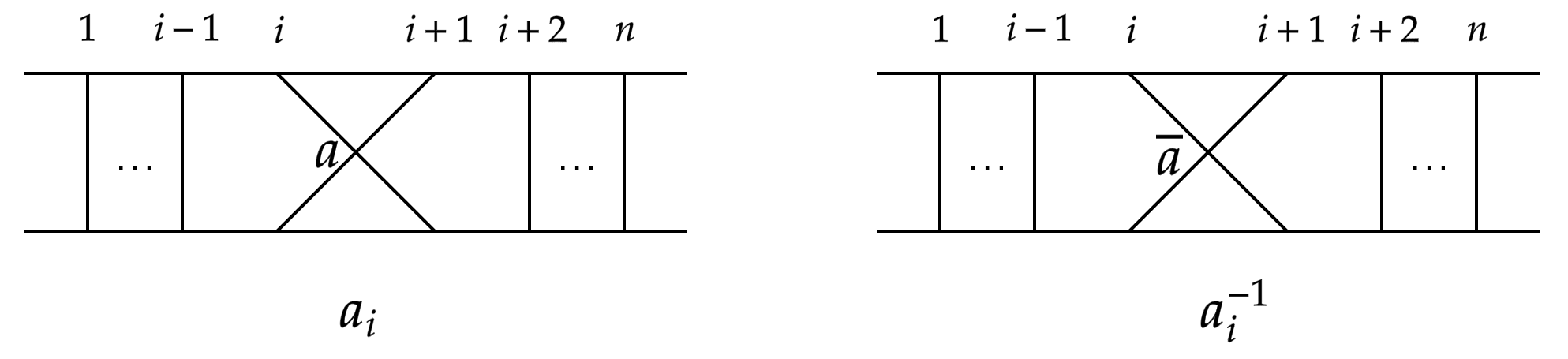}
    \caption{Elementary braids for the tag $a$}
    \label{Elementary-Braid}
\end{figure}

The kernel of the natural $\pi_n: gB_n \to S_n$ trailing end points of strands from top to bottom is called the {\it pure braid group} and is denoted by $gP_n$.\\

If there is some crossing type (tag) $x$ such that $\R_3(x, \overbar{x}, a)$ holds for some tag $a$, then we say that {\it $x$ dominates $a$}. If there is some crossing type $x$ such that $\R_3(\overbar{x},x,a)$ and $\R_3(x, \overbar{x}, a)$ holds for all tags $a$, then $x$ {\it dominates the theory}, see Figure \ref{x-Dominates}.

\begin{figure}[tph]
    \centering
    \includegraphics[width=0.3\linewidth]{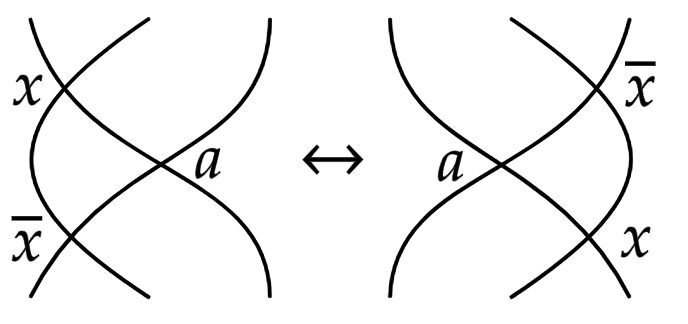}
    \caption{The tag $x$ dominates the tag $a$}
    \label{x-Dominates}
\end{figure}

A regular braid theory with a dominant tag say $x$ is called {\it normal}. 

\begin{example} Here are a few examples braid theories existing in the literature. 
\begin{itemize}
\item[(i)] {\it Artin braid group theory.}  The classical crossings $r$ and $\bar{r}$ have glyph of arc break depicted the over and under arcs in the diagram as shown in Figure \ref{Crossing-r}. It is a normal braid theory where both  $r$ and $\bar{r}$ dominates. 

\begin{figure}[tph]
    \centering
    \includegraphics[width=0.5\linewidth]{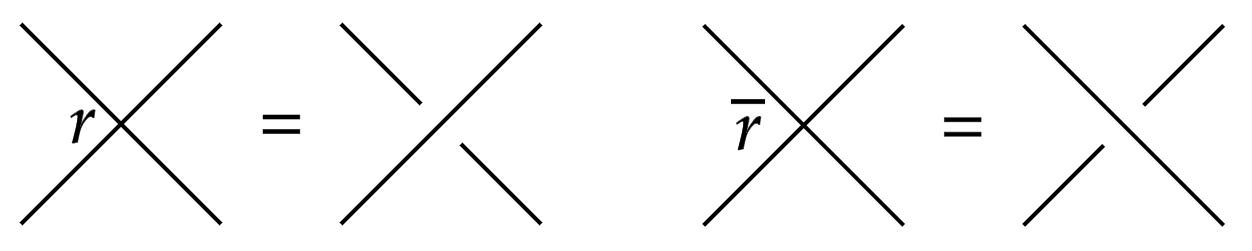}
    \caption{The positive and negative real crossings denoted by tag $r$ and $\bar{r}$}
    \label{Crossing-r}
\end{figure}

\item[(ii)] {\it Virtual braid group theory.} The virtual crossing type $v$ is depicted by glyph as shown in Figure \ref{Crossing-v}. The tag $v$ satisfy $R_2$-move and is involutive, hence $v$ is same as $\bar{v}$. It is a normal braid theory where $v$ is the dominant tag. It is to be noted that there is no other tag which dominates $v$.

\begin{figure}[tph]
    \centering
    \includegraphics[width=0.4\linewidth]{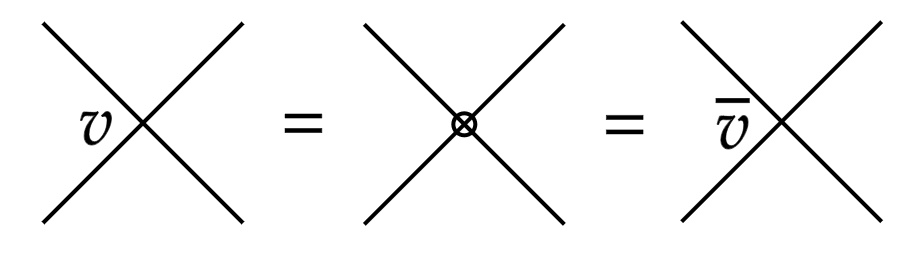}
    \caption{The glyph of virtual crossing with tag $v$}
    \label{Crossing-v}
\end{figure}

\item[(iii)] {\it Twin and virtual twin theory.} The real crossing type $t$, as shown in Figure \ref{Crossing-t} is involutive and satisfy $R_1$-move.  The twin braid group theory is regular but not normal. However, virtual twin group theory is normal with the crossing tag $v$ dominating all other tags. 

\begin{figure}[tph]
    \centering
    \includegraphics[width=0.3\linewidth]{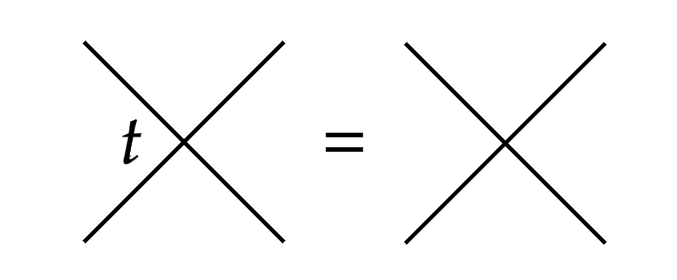}
    \caption{The twin crossing depicted by $t$}
    \label{Crossing-t}
\end{figure}
\end{itemize}
\end{example}

\begin{remark}
Apart from some examples given above, other known normal braid theories include welded $($or loop$)$ braids \cite{MR1257905}, unrestricted braids \cite{MR2128049}, flat braids \cite{MR2128049}, $($extended$)$ singular braids \cite{MR1654641}, virtual braids \cite{MR1721925} and virtual twins \cite{MR4027588}. Singular braid monoid \cite{MR1191478} is an example of non-regular braid theory.
\end{remark}

\begin{remark}
Throughout the paper, by braid theory we mean a normal generalized braid theory $($existence of dominant tag denoted by `x', there may not necessarily be only one $)$ and by braid group $gB_n$, we mean group associated to normal generalized braid theory on $n$ strands and the elementary braids $x_1, x_2, \dots , x_{n-1}$ are called the dominant generators. 
\end{remark}

\begin{remark}
The group $gB_n$ is generated by the set \[\{a_1, a_2, \dots, a_{n-1}, b_1, b_2, \dots, b_{n-1} , \dots, x_1, x_2, \dots, x_{n-1}, \dots \}\] for finitely many tags. Some of the following relations hold in $gB_n$:
\begin{eqnarray}
x_{i} x_{i+1} x_{i} &=& x_{i+1} x_{i} x_{i+1} \hspace*{5mm} \textrm{for } i = 1, 2, \dots, n-2, \label{1}\\ 
x_i x_j = &=& x_j x_i  \hspace*{5mm} \textrm{for } |i-j| >1, \label{2}\\
x_i a_j = &=& a_j x_i  \hspace*{5mm} \textrm{for } |i-j| >1 \textrm{ and for } i = 1, 2, \dots, n-2, \textrm{ and  for all tags, }\label{3}\\
x_{i} x_{i+1} a_{i} &=& a_{i+1} x_{i} x_{i+1} \hspace*{5mm} \textrm{for } i = 1, 2, \dots, n-2, \textrm{ and  for all tags. } \label{4}
\end{eqnarray}
\end{remark}

The following lemmas proved in  \cite{MR4474096} also hold for the braid theory. 
\begin{lemma}
In regular braid theory, the following statements are equivalent:
\begin{itemize}
\item[(i)] $x$ dominates $a$
\item[(ii)] $R_3(a, x, x)$ is allowed
\item[(iii)] $R_3(a, \bar{x}, \bar{x})$ is allowed
\item[(iv)] $R_3(x, a, x)$ is allowed
\item[(v)] $R_3(\bar{x}, a, \bar{x})$ is allowed.
\end{itemize}
\end{lemma}

The above lemma is depicted in Figure \ref{x-Dominates-Move}.
\begin{figure}[tph]
    \centering
    \includegraphics[width=0.6\linewidth]{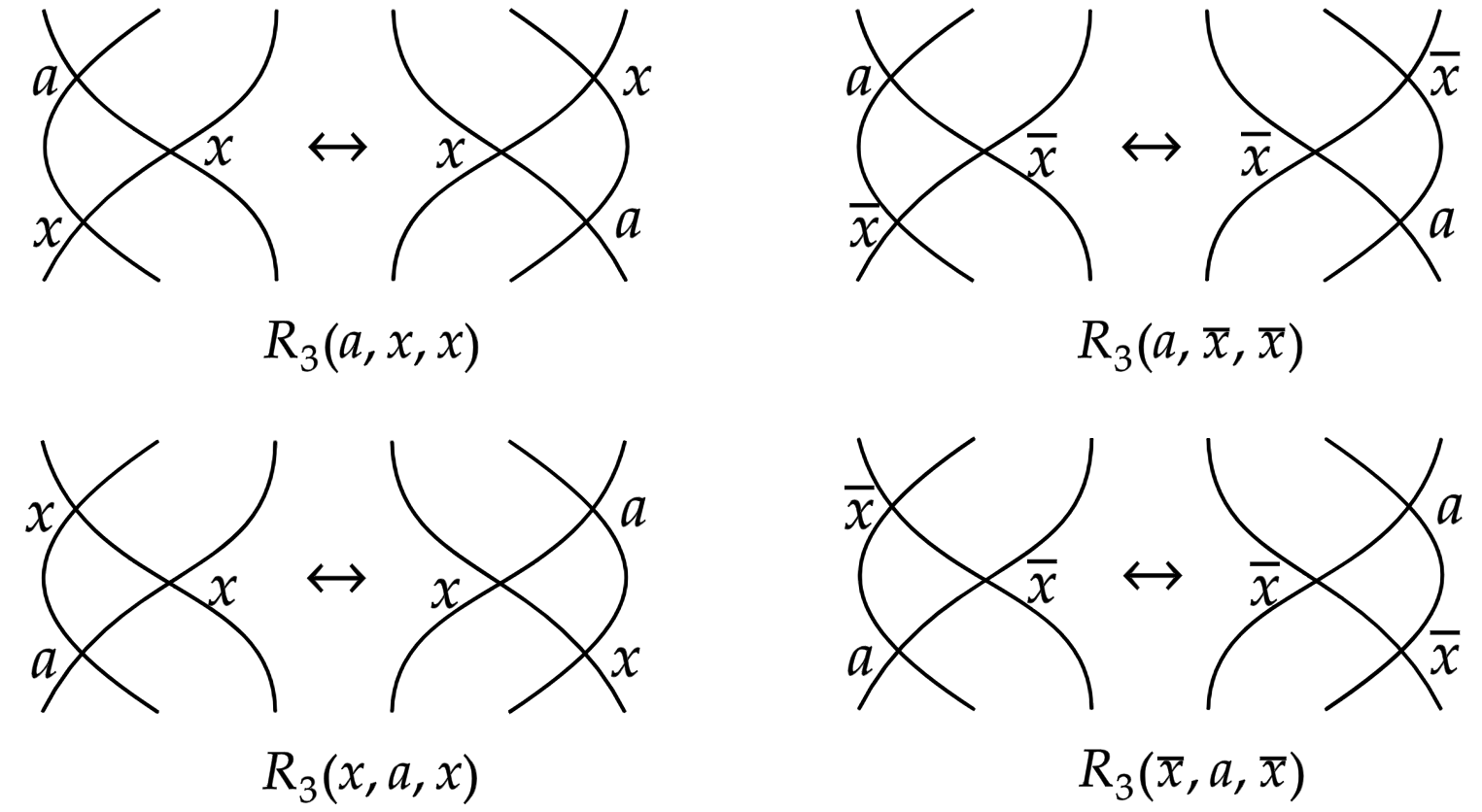}
    \caption{The allowed moves when the tag $x$ dominates the theory}
    \label{x-Dominates-Move}
\end{figure}

\begin{remark}
Let  $y_1, y_2, \dots y_q$ be arbitrary tags and the tags $x_1, x_2, x_3, x_3$ takes the values $x$ or $\bar{x}$ depending how the $R_2$ move is allowed. Then the following move holds in normal braid theory, which will be used throughout the paper. 

 \begin{figure}[tph]
    \centering
    \includegraphics[width=0.7\linewidth]{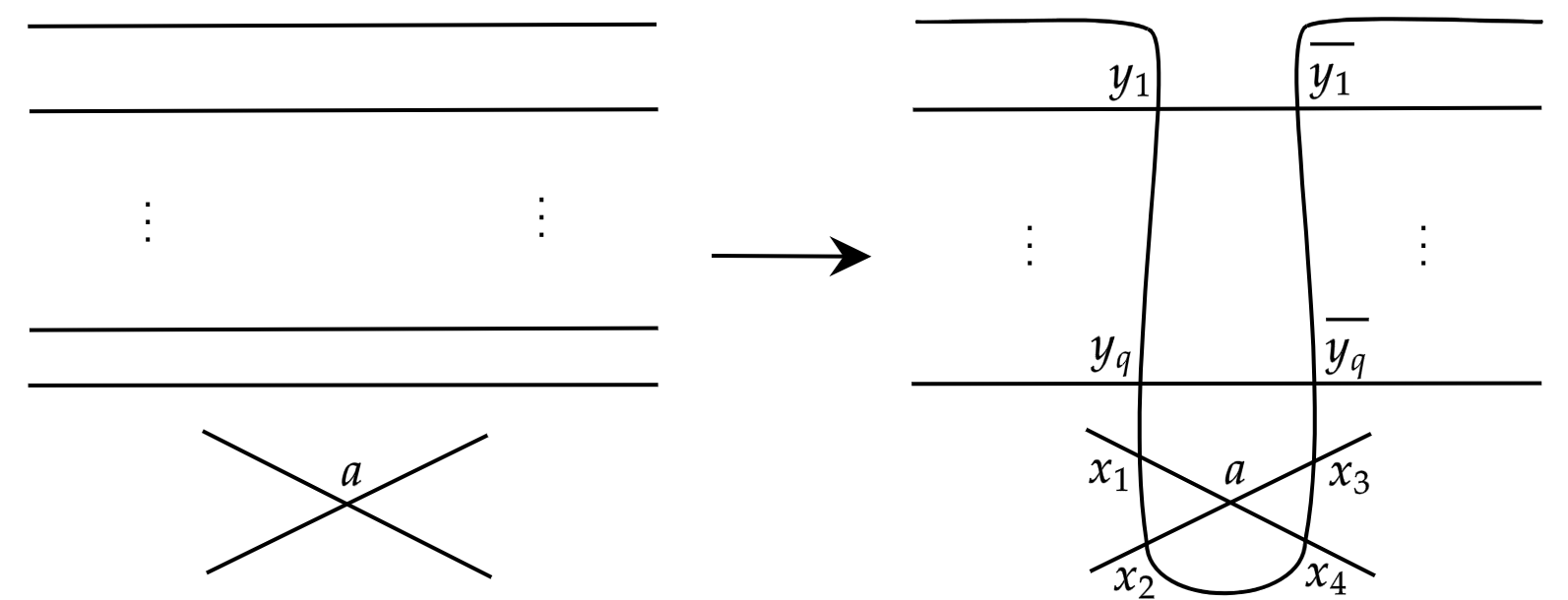}
    \caption{}
    \label{B-Move}
\end{figure}
\end{remark}

A subpath $P$ of a component of a braid diagram is said to be {\it $x$ above} if its end points are distinct from crossings and the only crossings it meets are of the two types illustrated below on the left of Figure \ref{x-Above-Below}. The portion of the subpath $P$ illustrated is drawn with a thicker line. Similarly $P$ is said to be {\it $x$ below} if the only crossings it meets are of the two types on the right of Figure  \ref{x-Above-Below}. The move in Figure \ref{B-Move} is used to define the {\it detour move} in the following lemma from  \cite{MR4474096}.

 \begin{figure}[tph]
    \centering
    \includegraphics[width=0.7\linewidth]{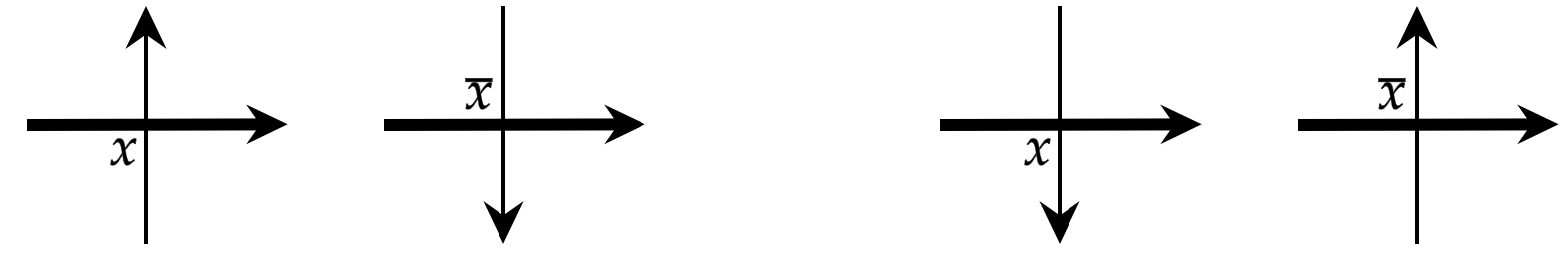}
    \caption{$x$ above and $x$ below paths}
    \label{x-Above-Below}
\end{figure}

\begin{lemma}[Detour move]
Let $P$ be an $x$ above/below subpath of a braided diagram $D$, where $x$ dominates the theory, and let $P'$ be a path with the same end points as $P$ which crosses $D$ in such a manner  as to create an $x$ above/below path. Then the diagrams $D$ and $(D-P) \cup P'$ are related by a sequence of $R$-moves.
\end{lemma}

\section{Generating set of pure braid group}\label{Pure-Generators-Section}

In this section, we give a generating set for pure braid group $gPB_n$. We shall use the generating set of $gB_n$ defined before and the Reidemeister-Schreier method \cite[Theorem 2.6]{MR207802}. For each $1 \leq k \leq n-1$, set 
\begin{equation*}
m_{k, i_k}:=
\begin{cases}
x_k x_{k-1}\dots x_{i_{k}+1} & \text{ for }  0 \leq i_k < k, \\
1 & \text{ for }    i_k = k, \\
\end{cases}
\end{equation*}
and 
\[\M_n := \big\{ m_{1, i_1}m_{2, i_2}\dots m_{n-1, i_{n-1}} ~| ~0 \leq i_k \leq k  \text{ for each } 1 \leq k \leq n-1   \big\}\]

as the Schreier system of coset representatives of $gPB_n$ in $gB_n$. For an element $w \in gB_n$, let $\overline{w}$ denote the unique coset representative of the coset of $w$ in the Schreier set $\M_n$. \\
By Reidemeister-Schreier method, the group $gPB_n$ is generated by the set
\[\big\{ \gamma( \mu, g) = (\mu g) (\overline{\mu g})^{-1} ~|~ \mu \in \M_n \text{ and } g \in \{a_1, \dots, a_{n-1}, b_1, \dots, b_{n-1}, \cdots, x_1, \dots, x_{n-1}, \cdots
\} \big\}.\]

We set
\begin{align*}
\vphantom{\lambda}_{a} \lambda_{i, i+1}&= a_i x_i,\\
\vphantom{\lambda}_{a} \lambda_{i+1, i}&= x_i a_i,\\
\vphantom{\lambda}_{x} \lambda_{i+1, i}&= x_i^2,
\end{align*}
for each $1 \le i \le n-1$ and 
\begin{align*}
\vphantom{\lambda}_{a}\lambda_{i,j} &= x_{j-1}^{-1} x_{j-2}^{-1} \dots x_{i+1}^{-1}\vphantom{\lambda}_{a} \lambda_{i, i+1} x_{i+1} \dots x_{j-2}  x_{j-1},\\
\vphantom{\lambda}_{a}\lambda_{j,i} &= x_{j-1}^{-1} x_{j-2}^{-1} \dots x_{i+1}^{-1}\vphantom{\lambda}_{a} \lambda_{i+1, i} x_{i+1} \dots x_{j-2}  x_{j-1},\\
\vphantom{\lambda}_{x}\lambda_{i,j} &= x_{j-1}^{-1} x_{j-2}^{-1} \dots x_{i+1}^{-1}\vphantom{\lambda}_{x} \lambda_{i, i+1} x_{i+1} \dots x_{j-2}  x_{j-1},
\end{align*}
for each $1 \leq i < j \leq n$ and $j \ne i+1$. Further, the tag $a$ means any crossing type in the braid theory corresponding to $gB_n$ except the dominant tag $x$. These notations will be used throughout this section.\\

Let \[\mathcal{S}= \big\{ \vphantom{\lambda}_{a}\lambda_{i, j}, \vphantom{\lambda}_{a}\lambda_{j, i}, \vphantom{\lambda}_{x} \lambda_{i,j} ~|~  1 \leq i< j \leq n  \text{ and all tags } a \text{ with dominant tag }  x \big\}.\]
\medskip

Let $\beta_i = \{a_i x_i^{-1}, x_i a_i \}$ for $i=1, 2, \dots, n-1$ and all tags $a$, and $X = \{ x_1, x_2, \dots, x_{n-1}\}$. We begin by proving the following result for the Reidemeister-Schreier generator.

\begin{lemma}\label{Generating-Set-Proof-Lemma-One}
The element  $\gamma( \mu, g)$ belongs to the set $\beta_i^{X}$, for $\mu \in \M_n$ and $g \in \{a_i, x_i \}$ for $i=1, 2, \dots, n-1$, all tags $a$ and dominant tag $x$.
\end{lemma}
\begin{proof}
We begin by claiming that the words in the set $\M_n$ are reduced words of the symmetric group $S_n$, generated by simple transpositions $\{ x_1, x_2, \dots, x_{n-1} \}$. The elements of $S_n$ can be viewed as end-points fixing homotopy classes of configurations of $n$ strands in $\mathbb{R} \times [0,1]$ connecting $n$ points on $\mathbb{R} \times \{1\}$ to $n$ points on $\mathbb{R} \times \{0\}$. It is well-known that a word $w \in S_n$ is reduced if and only if the length $l(w)$ of the word $w$ is equal to the number of inversions $\I(w)$ of the word $w$, where $\I(w)= | \{ i < j ~|~ w(i)> w(j)\}|$ . Consider $\rho_n = m_{1, i_1}m_{2, i_2}\dots m_{n-1, i_{n-1}} \in \M_n$. It is easy to check that the element $m_{1, i_1}$ is reduced. Let us suppose that the word $\rho_k=m_{1, i_1}m_{2, i_2}\dots m_{k, i_{k}}$ is reduced, and we prove that $\rho_{k+1}$ is reduced. Note that the element $\rho_k$ does not involve any strands from $k+1$ to $n$, as shown in Figure \ref{Permutation-Reduced}. So 
\[\I(\rho_{k+1}) = \I(\rho_{k})  + \text{ the } kth \text{ strand crossing the first } k \text{ strands from right to left}.\]
This implies 
\[\I(\rho_{k+1}) = l(\rho_{k}) + (k-i_k).\]
That is, 
\[\I(\rho_{k+1}) = \I(\rho_{k}) + (k-i_k)= l(\rho_{k+1}).\]
 \begin{figure}[tph]
    \centering
    \includegraphics[width=0.15\linewidth]{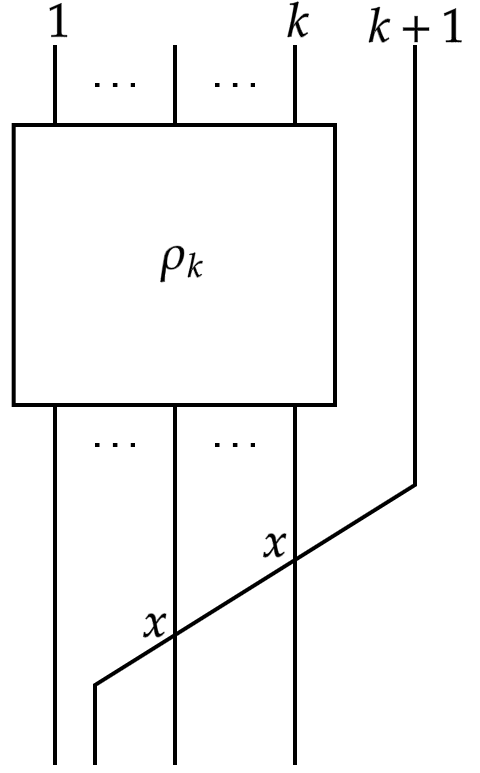}
    \caption{} 
    \label{Permutation-Reduced}
\end{figure}

This justifies our claim. Now let us consider $\mu= x_{i_1}x_{i_2} \dots x_{i_k}$ and \[\gamma( x_{i_1}x_{i_2} \dots x_{i_k}, a_i)=  (x_{i_1}x_{i_2} \dots x_{i_k} a_i) (\overline{x_{i_1}x_{i_2} \dots x_{i_k} a_i})^{-1}.\]
Suppose that $a_i$ does not get cancelled in the word $\overline{x_{i_1}x_{i_2} \dots x_{i_k} a_i}$, then $\overline{x_{i_1}x_{i_2} \dots x_{i_k} a_i}$ is reduced in $S_n$, and we have 
\[(x_{i_1}x_{i_2} \dots x_{i_k} a_i) (\overline{x_{i_1}x_{i_2} \dots x_{i_k} a_i})^{-1}= (x_{i_1}x_{i_2} \dots x_{i_k} a_i) (\overline{x_{i_1}x_{i_2} \dots x_{i_k} x_i})^{-1}.\]
Since, $x_{i_1}x_{i_2} \dots x_{i_k} x_i$ is reduced, then $x_{i_1}x_{i_2} \dots x_{i_k} x_i$ and $\overline{x_{i_1}x_{i_2} \dots x_{i_k} x_i}$ are related by relations \ref{1} and \ref{2} in $gB_n$ and $S_n$. So we have
\[\gamma( x_{i_1}x_{i_2} \dots x_{i_k}, a_i)=  x_{i_1}x_{i_2} \dots x_{i_k} (a_i x_i^{-1}) x_{i_k}^{-1} \ldots x_{i_2}^{-1} x_{i_1}^{-1}.\]
Therefore, $\gamma( x_{i_1}x_{i_2} \dots x_{i_k}, a_i) \in \beta_i^{X}$.\\
Now suppose that $a_i$ cancels with $x_{i_j}$ in the word $\overline{x_{i_1}x_{i_2} \dots x_{i_k} a_i}$, so we have 
\[\overline{x_{i_1}x_{i_2} \dots x_{i_k} a_i}= \overline{x_{i_1}x_{i_2} \dots x_{i_{j-1}} x_{i_{j+1}} \dots x_{i_k}}\]
which is a reduced word in $S_n$. Now $\overline{x_{i_1}x_{i_2} \dots x_{i_{j-1}} x_{i_{j+1}} \dots x_{i_k}}$ and ${x_{i_1}x_{i_2} \dots x_{i_{j-1}} x_{i_{j+1}} \dots x_{i_k}}$ differs by relations \ref{1} and \ref{2} in $gB_n$ and $S_n$.
\begin{align*}
\gamma( x_{i_1}x_{i_2} \dots x_{i_k}, a_i) &= {x_{i_1}x_{i_2} \dots x_{i_k} a_i} (\overline{x_{i_1}x_{i_2} \dots x_{i_k} a_i})^{-1} = x_{i_1}x_{i_2} \dots x_{i_k} a_i (\overline{x_{i_1}x_{i_2} \dots x_{i_{j-1}} x_{i_{j+1}} \dots x_{i_k}})^{-1}\\
&= x_{i_1}x_{i_2} \dots x_{i_{j-1}} x_{i_{j+1}} \dots x_{i_k} x_{i_j} a_i (\overline{x_{i_1}x_{i_2} \dots x_{i_{j-1}} x_{i_{j+1}} \dots x_{i_k}})^{-1}\\
&= (x_{i_1}x_{i_2} \dots x_{i_{j-1}} x_{i_{j+1}} \dots x_{i_k}) x_{i_j} a_i (\overline{x_{i_1}x_{i_2} \dots x_{i_{j-1}} x_{i_{j+1}} \dots x_{i_k}})^{-1}\\
&= (x_{i_1}x_{i_2} \dots x_{i_{j-1}} x_{i_{j+1}} \dots x_{i_k}) x_{i} a_i (\overline{x_{i_1}x_{i_2} \dots x_{i_{j-1}} x_{i_{j+1}} \dots x_{i_k}})^{-1}\\
&= (x_{i_1}x_{i_2} \dots x_{i_{j-1}} x_{i_{j+1}} \dots x_{i_k}) x_{i} a_i (x_{i_1}x_{i_2} \dots x_{i_{j-1}} x_{i_{j+1}} \dots x_{i_k})^{-1}.
\end{align*}

\end{proof}

\begin{lemma}\label{Generating-Set-Proof-Lemma-Two}
The set $\mathcal{S}^X \in \langle \mathcal{S}  \rangle$, that is, the conjugates of the set $\mathcal{S}$ by $X= \langle x_1, x_2, \dots, x_{n-1} \rangle$ belongs to the subgroup generated by $\mathcal{S}$.
\end{lemma}

\begin{proof}
We analyse the conjugation action of $x_k$ on the set $\mathcal{S}$. Let us first consider $\vphantom{\lambda}_{a} \lambda_{i,i+1}$, $\vphantom{\lambda}_{a} \lambda_{i+1, i}$ and $\vphantom{\lambda}_{x} \lambda_{i,i+1}$ for $i=1, 2, \dots, n-1$. 

\begin{itemize}
\item[(i)] If $ 1 \leq k \leq i-2$ or $i+2 \leq k \leq n-1$, then
\begin{align*}
x_k^{-1} \vphantom{\lambda}_{a} \lambda_{i,i+1} x_k &= \vphantom{\lambda}_{a} \lambda_{i,i+1}.\\
x_k^{-1} \vphantom{\lambda}_{a} \lambda_{i+1, i} x_k &= \vphantom{\lambda}_{a} \lambda_{i+1, i}.\\
x_k^{-1} \vphantom{\lambda}_{x} \lambda_{i,i+1} x_k &= \vphantom{\lambda}_{x} \lambda_{i,i+1}.
\end{align*}

\item[(ii)] If $ k= i-1$, then we have 
\begin{align*}x_k^{-1} \vphantom{\lambda}_{a} \lambda_{i,i+1} x_k &= \vphantom{\lambda}_{a} \lambda_{i-1,i+1},\\
x_k^{-1} \vphantom{\lambda}_{a} \lambda_{i+1, i} x_k &= \vphantom{\lambda}_{a} \lambda_{i+1, i-1},\\
x_k^{-1} \vphantom{\lambda}_{x} \lambda_{i,i+1} x_k &= \vphantom{\lambda}_{x} \lambda_{i-1,i+1},
\end{align*}
since
\begin{align*}
x_k^{-1} \vphantom{\lambda}_{a} \lambda_{i,i+1} x_k &= x_{i-1}^{-1} \vphantom{\lambda}_{a} \lambda_{i,i+1} x_{i-1}\\
&= x_{i-1}^{-1} a_i x_i  x_{i-1} = x_{i-1}^{-1} a_i x_{i-1}x_{i-1}^{-1} x_i  x_{i-1} = x_{i} a_{i-1} x_{i}^{-1} x_{i} x_{i-1} x_{i}^{-1}\\
&= x_{i} a_{i-1}x_{i-1} x_{i}^{-1} = \vphantom{\lambda}_{a} \lambda_{i-1,i+1}.
\end{align*}
 
\item[(iii)] If $k=i$, then
\begin{align*}
x_k^{-1} \vphantom{\lambda}_{a} \lambda_{i,i+1} x_k &= \vphantom{\lambda}_{x} \lambda_{i,i+1}^{-1} \cdot \vphantom{\lambda}_{a} \lambda_{i,i+1} \cdot \vphantom{\lambda}_{x} \lambda_{i,i+1}\\
x_k^{-1} \vphantom{\lambda}_{a} \lambda_{i+1, i} x_k &=  \vphantom{\lambda}_{a} \lambda_{i+1, i}\\
x_k^{-1} \vphantom{\lambda}_{x} \lambda_{i, i+1} x_k &=  \vphantom{\lambda}_{x} \lambda_{i, i+1}.
\end{align*}

\item[(iii)] If $k=i+1$, then
\begin{align*}x_k^{-1} \vphantom{\lambda}_{a} \lambda_{i,i+1} x_k &= \vphantom{\lambda}_{a} \lambda_{i,i+2},\\
x_k^{-1} \vphantom{\lambda}_{a} \lambda_{i+1, i} x_k &= \vphantom{\lambda}_{a} \lambda_{i+2, i},\\
x_k^{-1} \vphantom{\lambda}_{x} \lambda_{i,i+1} x_k &= \vphantom{\lambda}_{x} \lambda_{i,i+2},
\end{align*}

\end{itemize}
\medskip

Next, consider $\vphantom{\lambda}_{a} \lambda_{i,j}$, $\vphantom{\lambda}_{a} \lambda_{j, i}$ and $\vphantom{\lambda}_{x} \lambda_{i,j}$ for $ 1 \leq i < j \leq n-1$ and $j \neq i+1$. 

\begin{itemize}  
\item[(i)] If $1 \leq k \leq i-2$ or $j+1 \leq k \leq n-1$, then
\[x_k \lambda x_k ^{-1}= \lambda\]
for all $\lambda \in \mathcal{S}$. 
\medskip
\item[(ii)] For $k=i-1$, we have 
\begin{align*}
x_{i-1}^{-1} \vphantom{\lambda}_{a} \lambda_{i,j} x_{i-1} &=   {\vphantom{\lambda}_{x} \lambda_{i,j}} \cdot \vphantom{\lambda}_{a} \lambda_{i-1,j} \cdot \vphantom{\lambda}_{x} \lambda_{i,j}^{-1},\\
x_{i-1}^{-1} \vphantom{\lambda}_{a} \lambda_{j,i} x_{i-1} &=   {\vphantom{\lambda}_{x} \lambda_{i,j}} \cdot \vphantom{\lambda}_{a} \lambda_{j, i-1} \cdot \vphantom{\lambda}_{x} \lambda_{i,j}^{-1},\\
x_{i-1}^{-1} \vphantom{\lambda}_{x} \lambda_{i,j} x_{i-1} &=   {\vphantom{\lambda}_{x} \lambda_{i,j}} \cdot \vphantom{\lambda}_{x} \lambda_{i-1,j} \cdot \vphantom{\lambda}_{x} \lambda_{i,j}^{-1},
\end{align*}
since
\begin{align*}
x_{i-1}^{-1} \vphantom{\lambda}_{a} \lambda_{i,j} x_{i-1} &= x_{i-1}^{-1} x_{j-1}^{-1} x_{j-2}^{-1} \dots x_{i+1}^{-1} \vphantom{\lambda}_{a}  \lambda_{i, i+1} x_{i+1} \dots x_{j-2} x_{j-1}x_{i-1}\\
&= x_{i-1}^{-1} x_{j-1}^{-1} x_{j-2}^{-1} \dots x_{i+1}^{-1} a_i x_i x_{i+1}\dots x_{j-2} x_{j-1} x_{i-1} \\
&=  x_{j-1}^{-1} x_{j-2}^{-1} \dots  x_{i+1}^{-1} x_{i-1}^{-1} a_i x_i x_{i-1} x_{i+1}  \dots x_{j-2} x_{j-1}\\
&=  x_{j-1}^{-1} x_{j-2}^{-1} \dots  x_{i+1}^{-1} \underline{x_{i-1}^{-1} a_i x_{i-1}}~\underline{x_{i-1}^{-1}x_i x_{i-1}} x_{i+1}  \dots x_{j-2} x_{j-1}\\
&=  x_{j-1}^{-1} x_{j-2}^{-1} \dots  x_{i+1}^{-1} x_{i}a_{i-1} x_{i}^{-1} x_{i} x_{i-1} x_{i}^{-1} x_{i+1} \dots x_{j-2} x_{j-1}\\
&=  x_{j-1}^{-1} x_{j-2}^{-1} \dots  x_{i+1} ^{-1} x_{i}a_{i-1} x_{i-1} x_{i}^{-1} x_{i+1}  \dots x_{j-2} x_{j-1}\\
&=  (x_{j-1}^{-1} \dots  x_{i+1}^{-1} x_{i}^{2}x_{i+1}  \dots x_{j-1})(x_{j-1}^{-1} \dots  x_{i}^{-1} a_{i-1}x_{i-1}x_{i} \dots x_{j-1}) (x_{j-1}^{-1} \dots  x_{i+1}^{-1} x_{i}^{-2}x_{i+1}  \dots x_{j-1})\\
&= {\vphantom{\lambda}_{x} \lambda_{i,j}} \cdot \vphantom{\lambda}_{a} \lambda_{i-1,j} \cdot \vphantom{\lambda}_{x} \lambda_{i,j}^{-1}.
\end{align*}
\medskip
\item[(iii)] For $k=i$, we have
\begin{align*}
x_{i}^{-1} \vphantom{\lambda}_{a} \lambda_{i,j} x_{i} &=   \vphantom{\lambda}_{a} \lambda_{i+1,j},\\
x_{i}^{-1} \vphantom{\lambda}_{a} \lambda_{j,i} x_{i} &=   \vphantom{\lambda}_{a} \lambda_{j, i+1},\\
x_{i}^{-1} \vphantom{\lambda}_{x} \lambda_{i,j} x_{i} &= \vphantom{\lambda}_{x} \lambda_{i+1,j},
\end{align*}
since
\begin{align*}
x_{i}^{-1} \vphantom{\lambda}_{a} \lambda_{i,j} x_{i} &= x_{i}^{-1} x_{j-1}^{-1} x_{j-2}^{-1} \dots x_{i+1}^{-1} \vphantom{\lambda}_{a}  \lambda_{i, i+1} x_{i+1} \dots x_{j-2} x_{j-1}x_{i}\\
&= x_{i}^{-1} x_{j-1}^{-1} x_{j-2}^{-1} \dots x_{i+1}^{-1} a_i x_i x_{i+1}\dots x_{j-2} x_{j-1} x_{i} \\
&=  x_{j-1}^{-1} x_{j-2}^{-1} \dots  \underline{x_{i}^{-1} x_{i+1}^{-1} a_i}~\underline{ x_i x_{i+1} x_{i}}  \dots x_{j-2} x_{j-1}\\
&=  x_{j-1}^{-1} x_{j-2}^{-1} \dots  a_{i+1} x_{i}^{-1} x_{i+1}^{-1} x_{i+1}x_{i}x_{i+1}  \dots x_{j-2} x_{j-1}\\
&=  x_{j-1}^{-1} x_{j-2}^{-1} \dots x_{i+2}^{-1} a_{i+1}x_{i+1} x_{i+2}  \dots x_{j-2} x_{j-1}\\
&=  \vphantom{\lambda}_{a} \lambda_{i+1,j}.
\end{align*}

\item[(iv)] If $ i+1 \leq k \leq j-2$, then we have 
\begin{align*}
x_{k}^{-1} \vphantom{\lambda}_{a} \lambda_{i,j} x_{k} &=   \vphantom{\lambda}_{a} \lambda_{i,j},\\
x_{k}^{-1} \vphantom{\lambda}_{a} \lambda_{j,i} x_{k} &=   \vphantom{\lambda}_{a} \lambda_{j, i},\\
x_{k}^{-1} \vphantom{\lambda}_{x} \lambda_{i,j} x_{k} &= \vphantom{\lambda}_{x} \lambda_{i,j},
\end{align*}
since 
\begin{align*}
x_k^{-1}  \vphantom{\lambda}_{a} \lambda_{i,j} x_k &= x_k^{-1} x_{j-1}^{-1} \dots x_{k+1}^{-1} x_{k}^{-1} \dots x_{i+1}^{-1}  \vphantom{\lambda}_{a} \lambda_{i, i+1} x_{i+1} \dots x_k x_{k+1} \dots x_{j-1} x_k \\
&=  x_{j-1}^{-1} \dots x_k^{-1}x_{k+1}^{-1} x_{k}^{-1} \dots x_{i+1}^{-1}  \vphantom{\lambda}_{a} \lambda_{i, i+1} x_{i+1} \dots x_k x_{k+1} x_k
\dots x_{j-1}  \\
&=  x_{j-1}^{-1} \dots x_{k+1}^{-1}x_{k}^{-1} x_{k+1}^{-1} x_{k-1}^{-1} \dots x_{i+1}^{-1}  \vphantom{\lambda}_{a} \lambda_{i, i+1} x_{i+1} \dots x_{k-1} x_{k+1} x_{k} x_{k+1}
\dots x_{j-1}  \\
&=  x_{j-1}^{-1} \dots x_{k+1}^{-1}x_{k}^{-1} x_{k-1}^{-1} \dots x_{i+1}^{-1}  \vphantom{\lambda}_{a} \lambda_{i, i+1} x_{i+1} \dots x_{k-1} x_{k} x_{k+1}
\dots x_{j-1}  \\
&=  \vphantom{\lambda}_{a} \lambda_{i,j}.
\end{align*}

\item[(v)] If $k= j-1$, then 
\begin{align*}
x_k^{-1}  \vphantom{\lambda}_{a} \lambda_{i,j} x_k  &= \vphantom{\lambda}_{x} \lambda_{j-1, j}^{-1}\cdot  \vphantom{\lambda}_{a} \lambda_{i,j-1} \cdot \vphantom{\lambda}_{x} \lambda_{j-1, j}.\\
x_k^{-1}  \vphantom{\lambda}_{a} \lambda_{j,i} x_k  &= \vphantom{\lambda}_{x} \lambda_{j-1,j}^{-1}\cdot  \vphantom{\lambda}_{a} \lambda_{j-1, i} \cdot \vphantom{\lambda}_{x} \lambda_{j-1, j}.\\
x_k^{-1}  \vphantom{\lambda}_{x} \lambda_{i,j} x_k  &= \vphantom{\lambda}_{x} \lambda_{j-1, j}^{-1}\cdot  \vphantom{\lambda}_{x} \lambda_{i,j-1} \cdot \vphantom{\lambda}_{x} \lambda_{j-1, j}.
\end{align*}

\item[(vi)] If $k= j$, then 
\begin{align*}
x_k^{-1}  \vphantom{\lambda}_{a} \lambda_{i,j} x_k  &=   \vphantom{\lambda}_{a} \lambda_{i,j+1}.\\
x_k^{-1}  \vphantom{\lambda}_{a} \lambda_{j,i} x_k  &=   \vphantom{\lambda}_{a} \lambda_{j+1, j}.\\
x_k^{-1}  \vphantom{\lambda}_{x} \lambda_{i,j} x_k  &=   \vphantom{\lambda}_{x} \lambda_{i,j+1}.
\end{align*}
\end{itemize}
\end{proof}

We now state the main theorem of this section.

\begin{theorem}\label{gPB_n-generators}
The pure normal generalized braid group $gPB_n$ on  $n \ge 2 $ strands is generated by \[\mathcal{S}= \big\{ \vphantom{\lambda}_{a}\lambda_{i, j}, \vphantom{\lambda}_{a}\lambda_{j, i}, \vphantom{\lambda}_{x} \lambda_{i,j} ~|~  1 \leq i< j \leq n  \text{ and all tags } a \text{ with dominant tag }  x \big\}.\]
\end{theorem}

\begin{proof}
By Reidemeister-Schreier method, the group $gPB_n$ is generated by the set
\[\big\{ \gamma( \mu, g) = (\mu g) (\overline{\mu g})^{-1} ~|~ \mu \in \M_n \text{ and } g \in \{a_1, \dots, a_{n-1}, b_1, \dots, b_{n-1}, \cdots, x_1, \dots, x_{n-1}, \cdots
\} \big\}.\]
We first note that $\beta_i$ is a subset of group generated by $\mathcal{S}$. Then, by Lemmas \ref{Generating-Set-Proof-Lemma-One} and \ref{Generating-Set-Proof-Lemma-Two} , we have $$ \gamma( \mu, g) \in \beta_i^X \subset \langle \mathcal{S} \rangle^X \subset \langle \mathcal{S} \rangle,$$ for some $i \in \{1, 2, \dots, n-1 \}$. Geometrically, the generators look as in Figure \ref{Pure-Generator}.
 \begin{figure}[tph]
    \centering
    \includegraphics[width=0.9\linewidth]{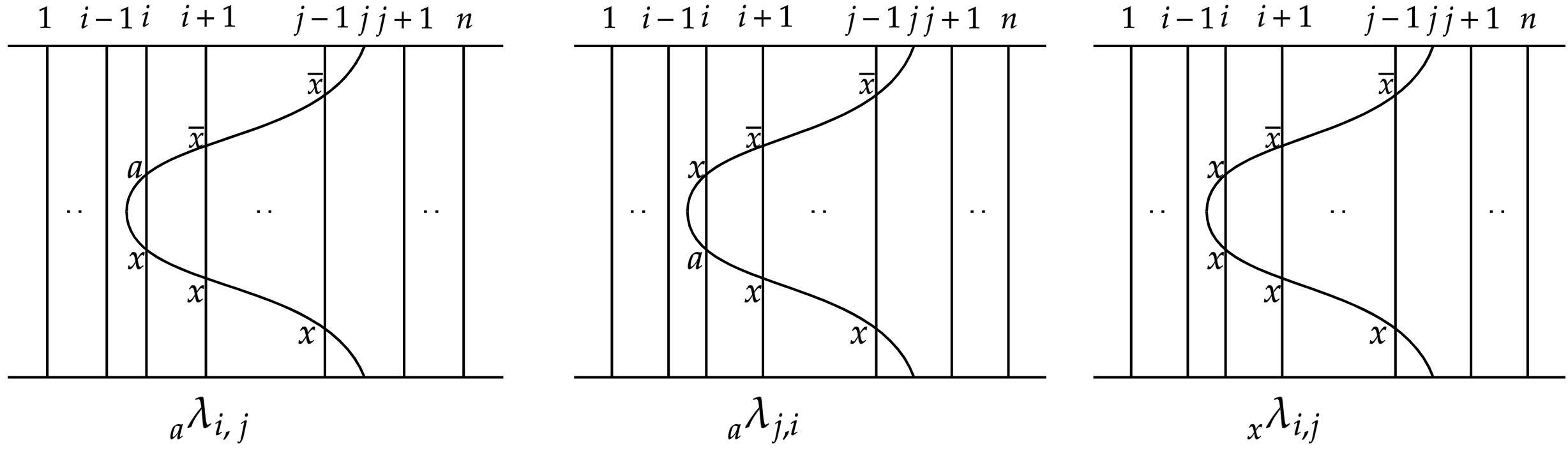}
    \caption{The generators $ \vphantom{\lambda}_{a}\lambda_{i, j}, \vphantom{\lambda}_{a}\lambda_{j, i}$ and $ \vphantom{\lambda}_{x} \lambda_{i,j} $ for all tags $a$ and dominant tag $x$.} 
    \label{Pure-Generator}
\end{figure}
\end{proof}

\section{Quasitoric generalized braids}\label{Quasitoric-Braid-Section}

In this section, we introduce the definition of a quasitoric generalized braid. \\
Let $p,q$ be positive integers. A generalized braid $\beta$ is said to be {\it quasitoric of type $(p,q)$} if it can be expressed as $\beta= \beta_1 \beta_2 \ldots \beta_q$, where $\beta_j=y_{j,p-1} y_{j,p-2} \ldots y_{j,1}$, where $y_{j,i} \in \{a_i, \overbar{a_i}, b_i, \overbar{b_i}, \ldots \}$. A quasitoric generalized braid of type $(p,q)$ is shortly called {\it $p$-quasitoric braid}. For example, see Figure \ref{Example-Quasitoric-Braid}.

 \begin{figure}[tph]
    \centering
    \includegraphics[width=0.15\linewidth]{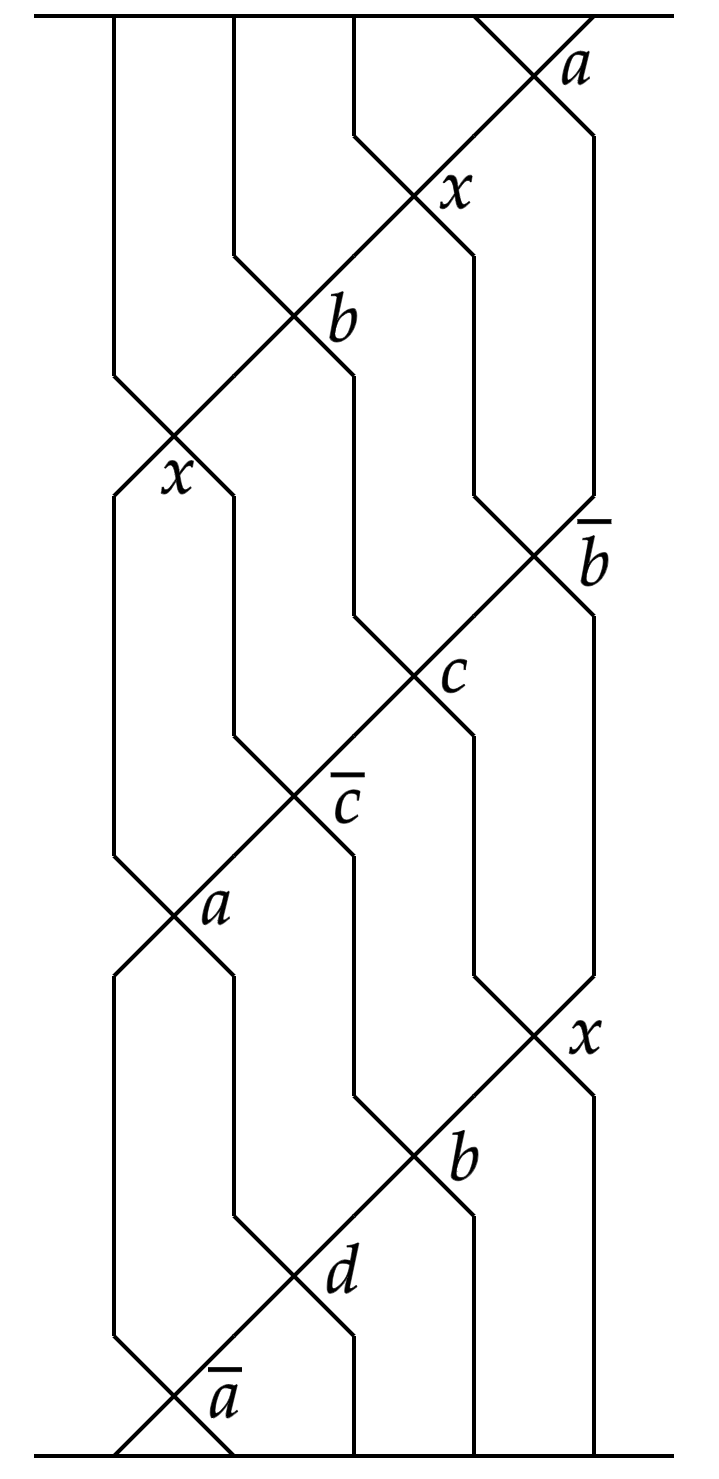}
    \caption{A $5$-quasitoric braid given by $\beta=(a_4 x_3 b_2 x_1)(b_4^{-1} c_3 c_2^{-1} a_1)(x_4 b_3 d_2 a_1^{-1})$}
    \label{Example-Quasitoric-Braid}
\end{figure}

For positive integers $i,j$ with $1 \leq i < j \leq n$, an $n$ strand braid $\beta$ is called {\it $(i,j)$-quasitoric braid with $n$ strands} if it has a braid diagram of the form shown in Figure \ref{(i,j)-Quasitoric-Braid}, where $\beta'$ is a $(j-i+1)$-quasitoric braid.

 \begin{figure}[tph]
    \centering
    \includegraphics[width=0.2\linewidth]{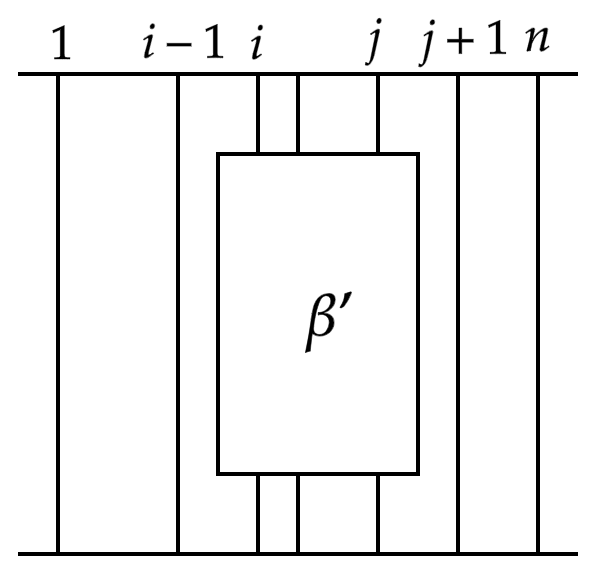}
    \caption{$(i,j)$-quasitoric braid on $n$ strands} 
    \label{(i,j)-Quasitoric-Braid}
\end{figure}

\begin{remark}
If a $(p,q)$-quasitoric braid $\beta$ is pure, then it is easy to prove that $q$ is a multiple of $p$.
\end{remark}

Let $qgB_n$ be the set of all quasitoric generalized braids on $n$ strands. We first prove that the identity element of the braid group $gB_n$ is quasitoric.

\begin{lemma}\label{lem:id_is_quasitoric}
For all $n \geq 2$, the identity element of the braid group $gB_n$ can be expressed as an $n$-quasitoric braid.
\end{lemma}
\begin{proof}
Fix a dominant tag (say $x$) in the braid theory corresponding to the group $gB_n$. Consider shadow of a quasitoric braid of type $(n,n)$ without any tags on $n$ strings as shown in the left of Figure \ref{Proof-Id-Quasitoric}. Our strategy is to mark tags $x, \overbar{x}$ in such a way that it yields the trivial braid through a sequence of detour moves. To begin with, we consider the $n$th strand, mark the tag $x$ on the crossing when the strand goes from right to left, and $\overbar{x}$ on the crossings when the strand goes from left to right. Next, we consider the $(n-1)$th strand, mark the tag $x$ on the untagged crossings when the $(n-1)$th strand goes from right to left. Similarly, mark the tag $\overbar{x}$ on the untagged crossings when $(n-1)$th strand goes from left to right. Inductively repeat the process to get a quasitoric braid $\beta$ of type $(n,n)$ in $gB_n$. Note that we can straighten the $n$th strand by a detour move. Now after straightening the $n$th stand, observe that all the crossings are of type $x$ when the $(n-1)$th strand goes from right to left and are of type $\overbar{x}$ when the $(n-1)$th strand goes from left to right. Thus we can again apply detour move on $(n-1)$th strand and straighten it. Repeating this process for all strands, we get that the quasitoric braid $\beta$ represents the identity element in $gB_n$.\end{proof}
 \begin{figure}[tph]
    \centering
    \includegraphics[width=0.6\linewidth]{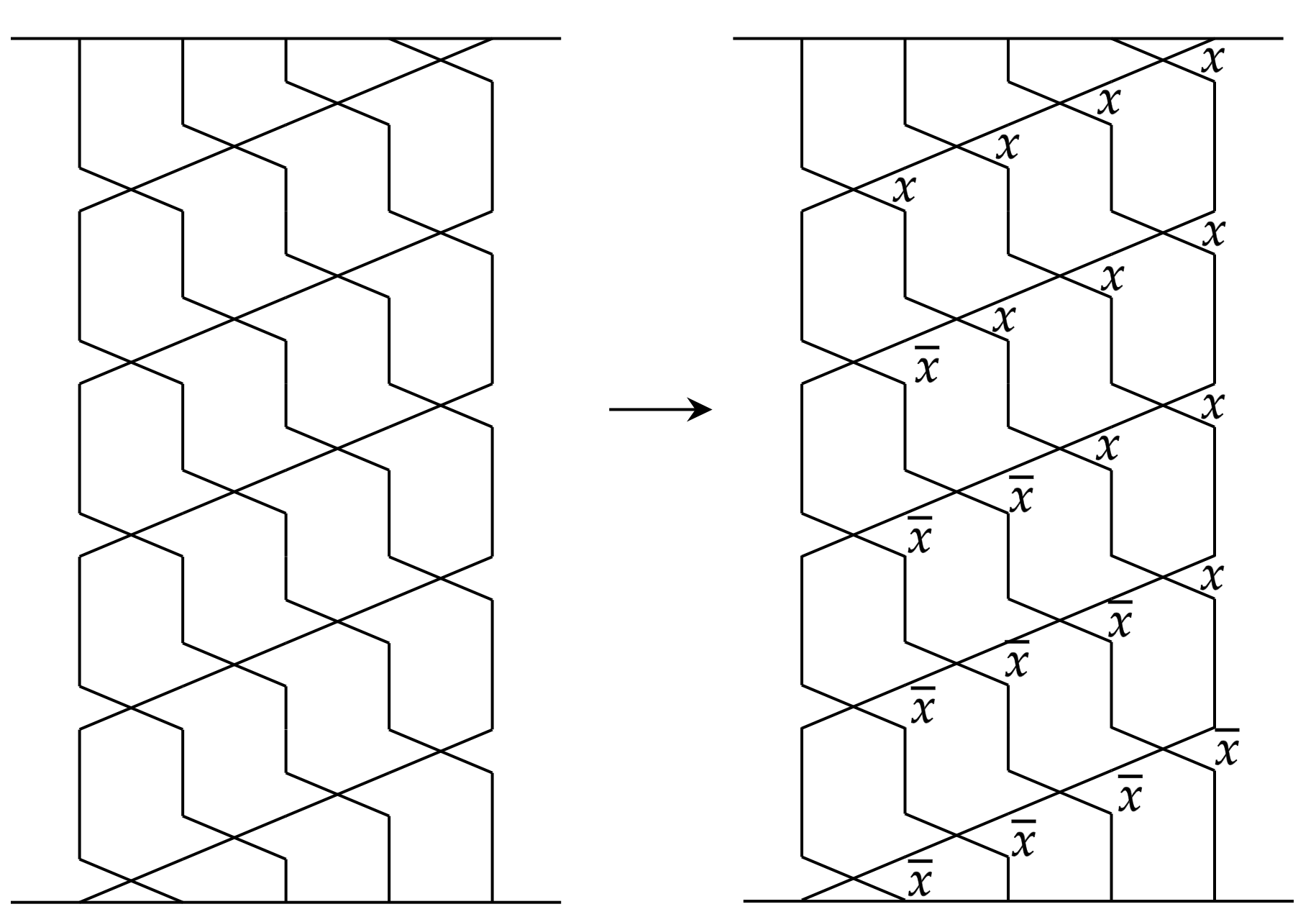}
    \caption{} 
    \label{Proof-Id-Quasitoric}
\end{figure}

\begin{remark}
It is not difficult to prove that the quasitoric braid $\beta$ given by \[(x_{n-1} x_{n-2} \dots x_1) (x_{n-1} x_{n-2} \dots x_1^{-1}) \cdots (x_{n-1} x_{n-2} \dots x_j^{-1} x_{j-1}^{-1} \dots  x_1^{-1}) \cdots (x_{n-1}^{-1} x_{n-2}^{-1} \dots x_1^{-1})\] represents the identity element in the group $gB_n$. A particular example on five strands is shown in Figure \ref{Proof-Id-Quasitoric}.
\end{remark}

\begin{lemma}\label{Lemma-(i,j)-Pure-Quasitoric}
For integers $i,j$ with $1 \leq i < j \leq n$, every $(i,j)$-quasitoric pure braid with $n$ strands is $n$-quasitoric.
\end{lemma}

\begin{proof}
It suffices to show that $(1, n-1)$-quasitoric pure braid and $(2,n)$-quasitoric pure braid with $n$ strands is $n$-quasitoric. Consider the $(n-1)$-quasitoric pure braid diagram $\beta$ in $(1, n-1)$-quasitoric pure braid diagram with $n$ strands as shown in left of Figure  \ref{(i,j)-Pure-Quasitoric-1}. We then slide the $n$th strand parallel to $(n-1)$th strand into $\beta$ via detour moves and dominant tags $x$ and $\bar{x}$, as shown in Figure \ref{(i,j)-Pure-Quasitoric-1}. We then by a sequence of detour moves obtain a $n$-quasitoric braid as shown as a particular example in Figure \ref{(i,j)-Pure-Quasitoric-2}. \\
Similarly, we can slide the first strand parallel to second strand in $(2,n)$-quasitoric pure braid to get a $n$-quasitoric braid, and this completes the proof.

\begin{figure}[tph]
    \centering
    \includegraphics[width=0.45\linewidth]{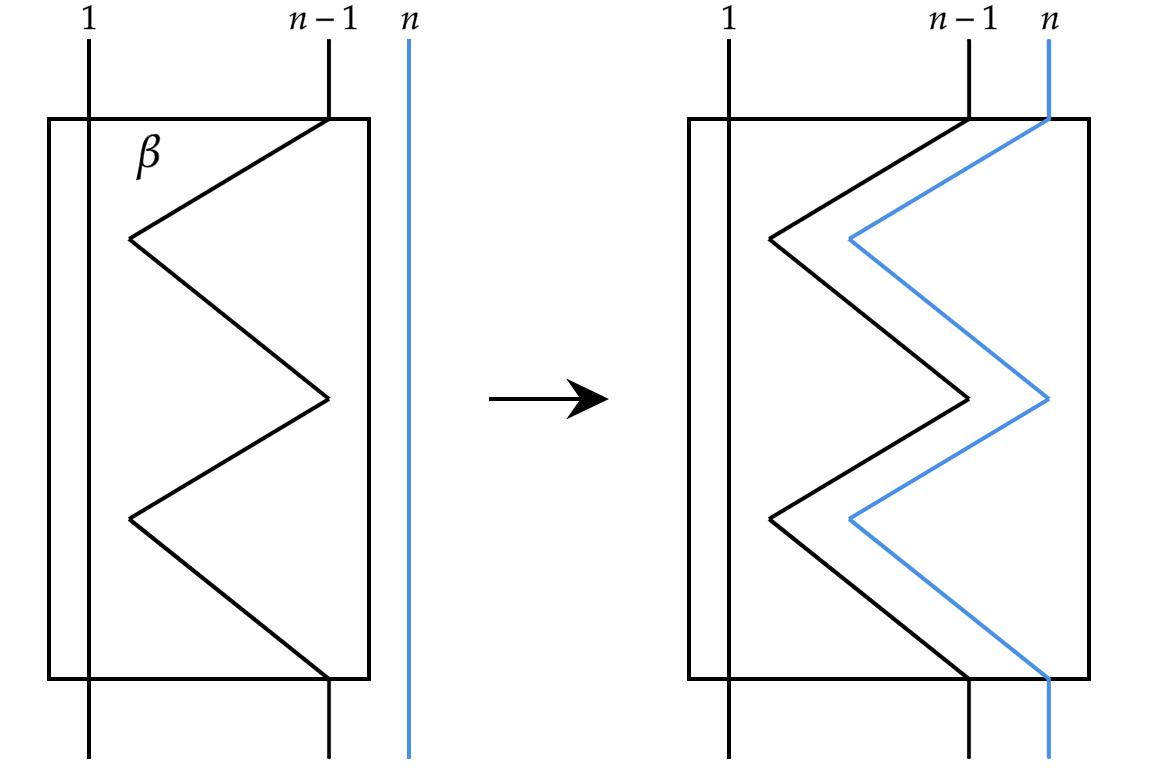}
    \caption{Slide $n$th strand parallel to $(n-1)$th strand via dominant tag $x$ and $\bar{x}$} 
    \label{(i,j)-Pure-Quasitoric-1}
\end{figure}

\begin{figure}[tph]
    \centering
    \includegraphics[width=0.75\linewidth]{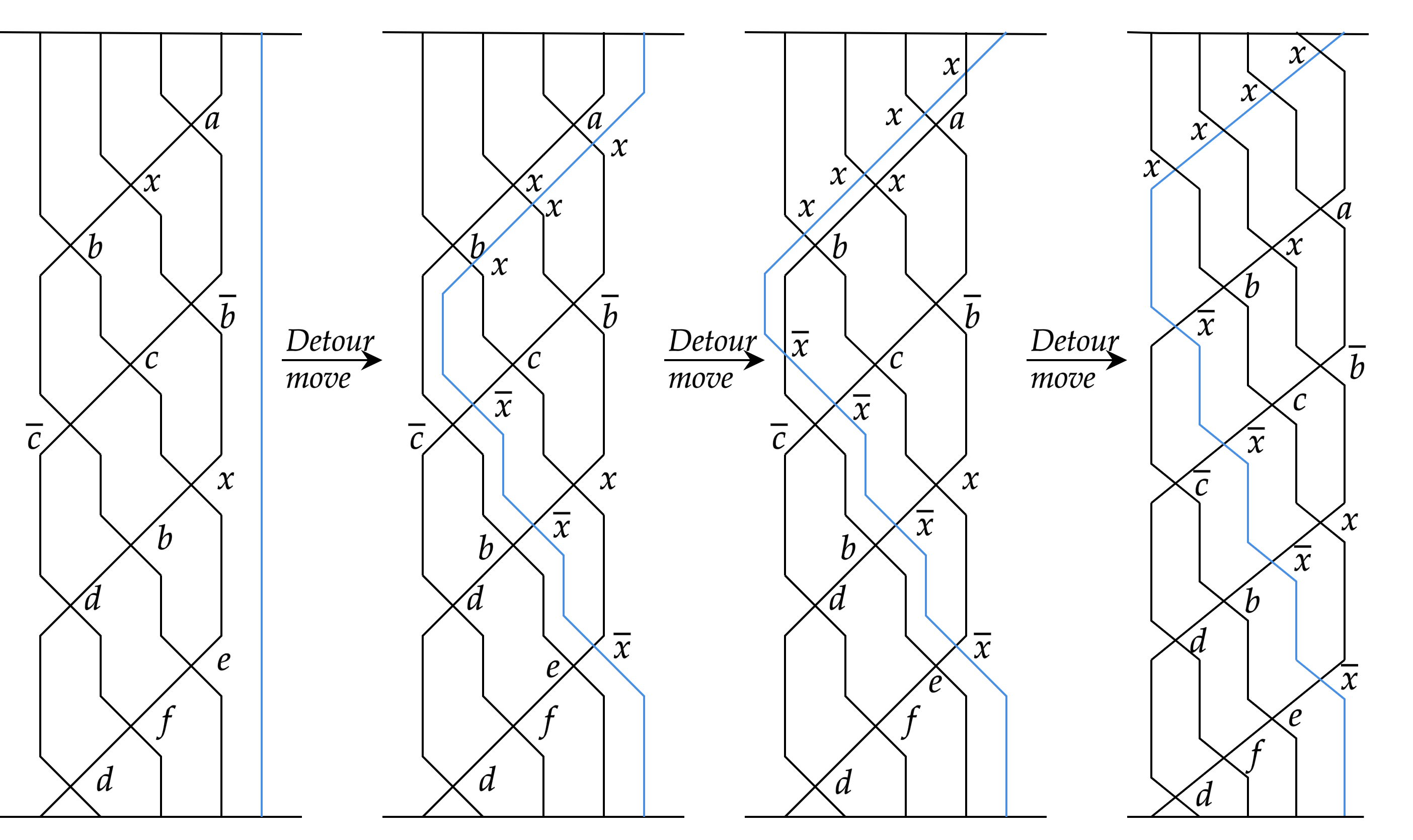}
    \caption{} 
    \label{(i,j)-Pure-Quasitoric-2}
\end{figure}
\end{proof}
We use the generating set obtained for pure braid group, and prove that every pure normal generalized braid is quasitoric.
\begin{theorem}\label{thm:pure_is_quasitoric}
Every pure normal generalized braid is quasitoric.
\end{theorem}

\begin{proof}
It suffices to show that pure braids $\vphantom{\lambda}_{a} \lambda_{i,j}^{\pm 1}$, $\vphantom{\lambda}_{a} \lambda_{j, i}^{\pm 1}$ and $\vphantom{\lambda}_{x} \lambda_{i,j}^{\pm 1}$ for $ 1 \leq i < j \leq n-1$ are quasitoric. Let us first consider $\vphantom{\lambda}_{a} \lambda_{i,j}$.
\begin{figure}[tph]
    \centering
    \includegraphics[width=0.75\linewidth]{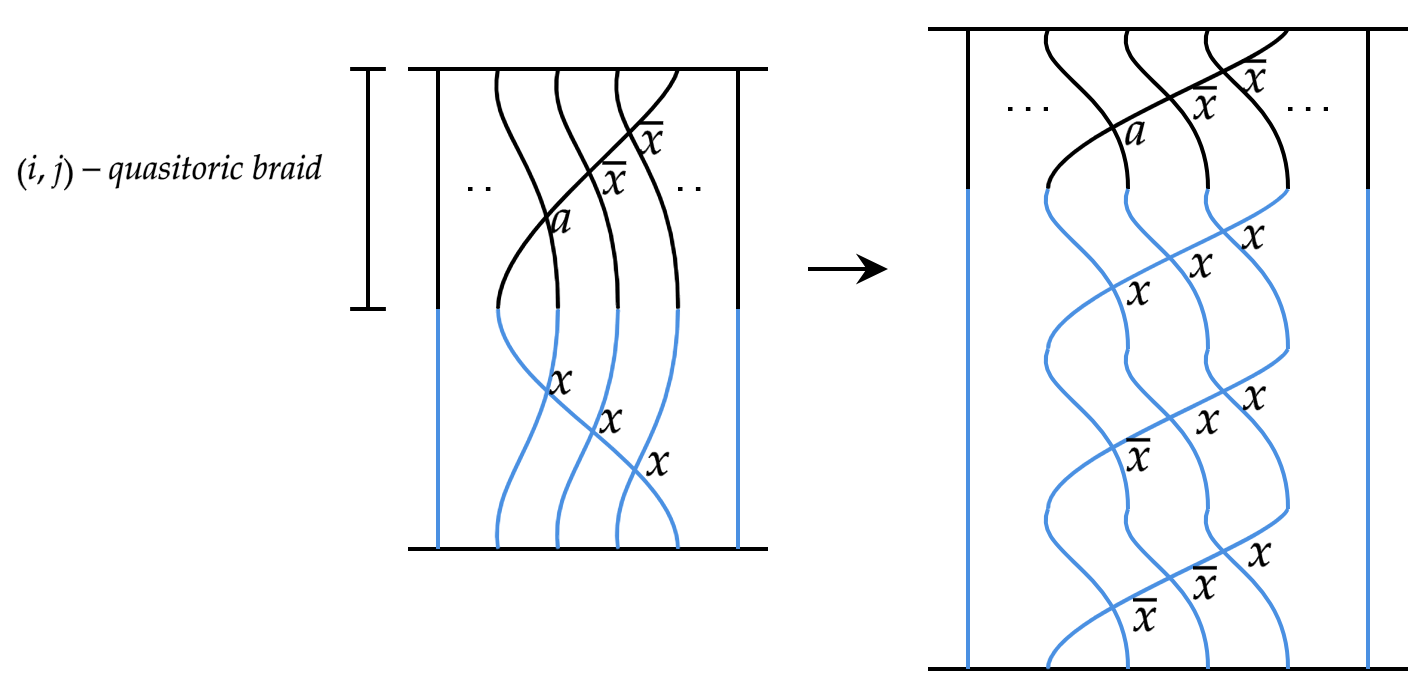}
    \caption{} 
    \label{Pure-Quasitoric-Proof-1}
\end{figure}
 In the Figure \ref{Pure-Quasitoric-Proof-1}, the upper part of the generator is $(i,j)$-quasitoric with $n$ strands, hence $n$-quasitoric, whereas the bottom part of the generator can be made $(i,j)$-quasitoric by the moves shown. In other words, the generator $\vphantom{\lambda}_{a} \lambda_{i,j}$ can be expressed as 
\begin{align*}
\vphantom{\lambda}_{a} \lambda_{i,j}&= x_{j-1}^{-1} x_{j-2}^{-1} \dots x_{i+1}^{-1}a_i x_i x_{i+1} \dots x_{j-2}  x_{j-1} \\
&= (x_{j-1}^{-1} \dots x_{i+1}^{-1}a_i) (x_{j-1} x_{j-2} \dots x_i) (x_{j-1} x_{j-2} \dots x_i^{-1}) \cdots (x_{j-1} \dots x_kx_{k-1}^{-1} \dots x_i^{-1}) \cdots  (x_{j-1} x_{j-2}^{-1} \dots x_i^{-1}).
\end{align*}

Then, by Lemma \ref{Lemma-(i,j)-Pure-Quasitoric}, we get that the generator $\vphantom{\lambda}_{a} \lambda_{i,j}$ is quasitoric. The proof of generator $\vphantom{\lambda}_{a} \lambda_{j,i}$ being quasitoric is given in Figure \ref{Pure-Quasitoric-Proof-2}. Similarly, the rest of the cases can be dealt in the same manner. \\

\begin{figure}[tph]
    \centering
    \includegraphics[width=0.8\linewidth]{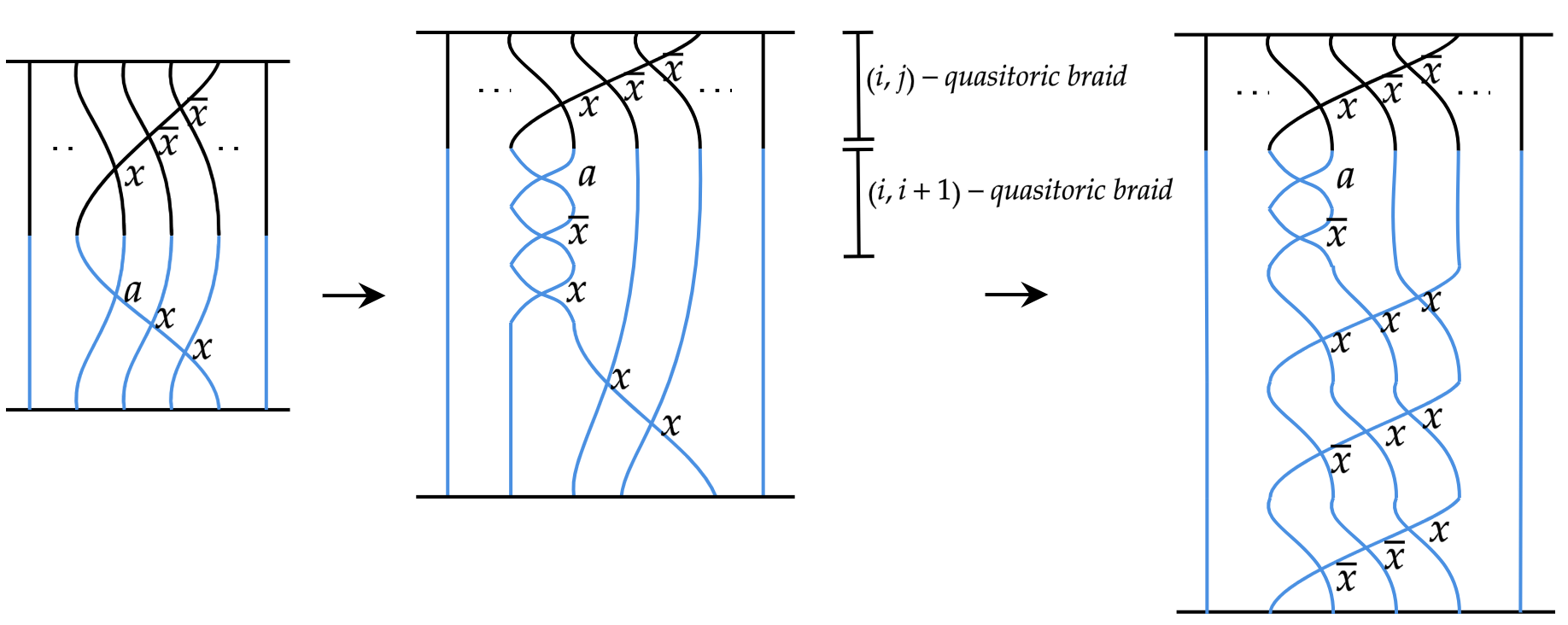}
    \caption{} 
\label{Pure-Quasitoric-Proof-2}
\end{figure}
\end{proof}

We recall the notion of regular generalized knot theory defined in \cite{MR4474096}. A knot diagram is an immersion of finitely circles into the $2$-sphere, where the double points are either decorated with tags denoted by roman letters `$a$' or their negative versions `$\bar{a}$'. We assume that there is at least one tag `$a$' for which $R_1(a)$ move is allowed, shown in Figure \ref{R1(a)-move}. Two knot diagrams are {\it isotopic} if they are related by a finite sequence moves shown in figures \ref{R-Moves} and \ref{R1(a)-move}, which are predetermined for the given tags. In \cite{MR4474096}, it is proved that every regular generalized knot is a closure of a regular generalized braid diagram. Consequently, note that the braids related by $M_1$ and $M_2$ moves shown in Figure \ref{M1 and M2 moves} have isotopic closures. 
\begin{figure}[tph]
    \centering
    \includegraphics[width=0.4\linewidth]{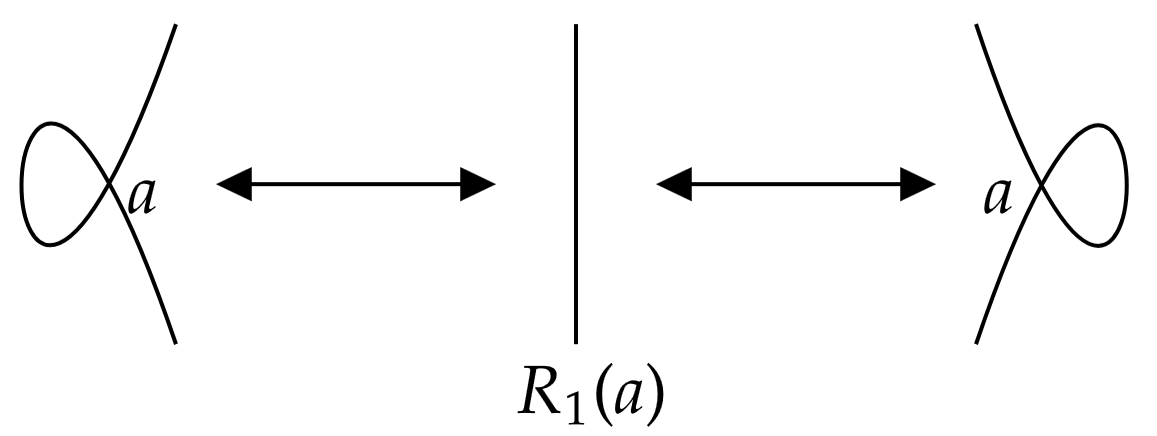}
    \caption{$R_1(a)$ move} 
     \label{R1(a)-move}
\end{figure}

\begin{figure}[tph]
    \centering
    \includegraphics[width=0.5\linewidth]{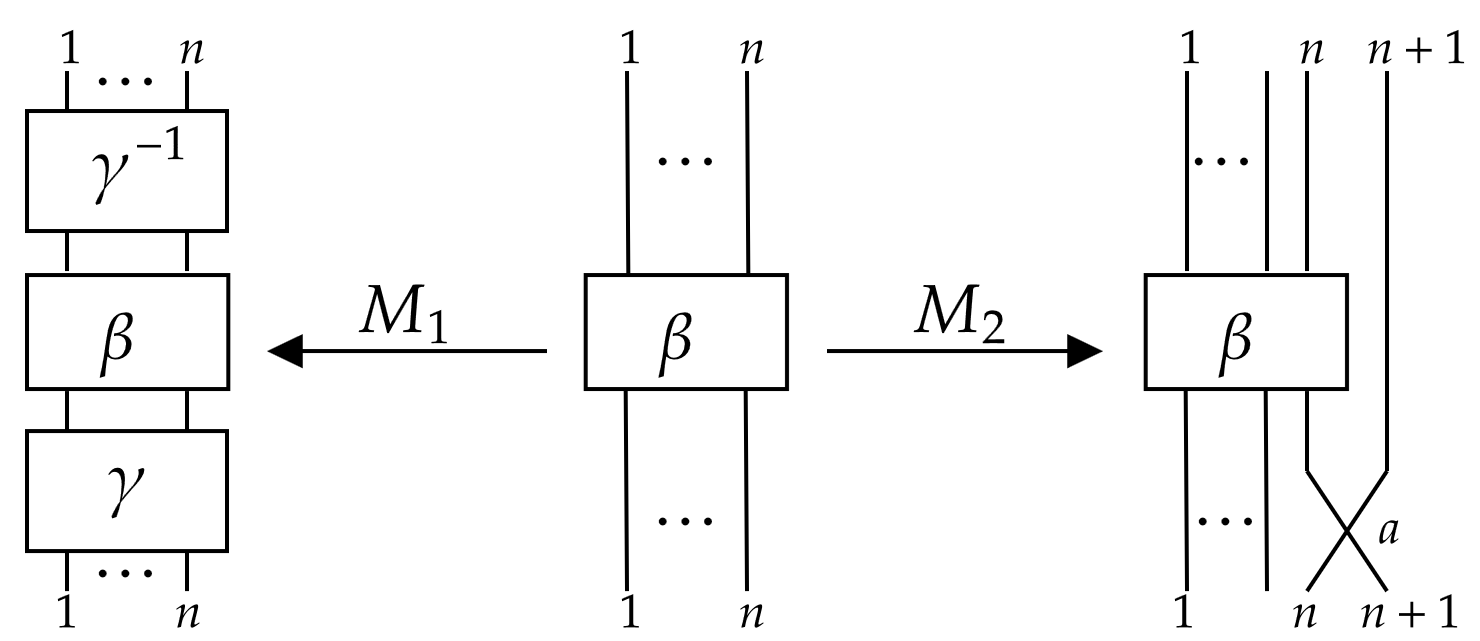}
    \caption{$M_1$ and $M_2$ moves} 
     \label{M1 and M2 moves}
\end{figure}
We now prove the main theorem of this section. 

\begin{theorem}\label{Quasitoric-Main-Theorem}
Every oriented normal generalized link is a closure of a quasitoric normal generalized braid.
\end{theorem}
\begin{proof}
Let $L$ be a normal generalized link. Then by \cite[Theorem 6.1]{MR4474096}, for some $n \geq 1$, there exists a braid $\beta \in gB_n$ whose closure is equivalent to $L$. Consider the permutation $\pi_n(\beta)$ and the orbits of the action of $\pi_n(\beta)$ on $\{1,2, \ldots, n\}$. The number of elements in each orbit might be different.

The $M_1$ move on a braid $\beta$ conjugates it, and therefore, conjugates the corresponding permutation. Conjugating a permutation only shuffles the elements among the orbits but does not change the number of elements in each orbit. On the other hand, applying the $M_2$ move on a braid, adds one new strand in the braid and adds $n+1$ to the orbit containing $n$ in the corresponding permutation and other orbits remain unchanged.

Thus by reiterating the $M_1$ and $M_2$ moves on $\beta$, it is ensured that there is a braid $\beta'$ in $gB_m$ for some $m \geq n$, such that the closure of $\beta'$ is equivalent to $L$ and the permutation $\pi_m(\beta')$ is the $k$th power of the cyclic permutation $(1~ 2~ \cdots m)$ for some non-negative integer $k$. Now observe that the braid $\beta''$, where
\[ \beta''= \beta' (x_{m-1} x_{m-2} \ldots x_1)^{-k},\]
is a pure braid, which is quasitoric by Theorem \ref{thm:pure_is_quasitoric}. Thus, $\beta'$ being a product of two quasitoric braids is quasitoric.
\end{proof}

\begin{theorem}\label{Subgroup-Theorem}
For $n \geq 2$, the set $qgB_n$ of all quasitoric normal generalized braids on $n$ strands forms a subgroup of the group $gB_n$ under the operation of concatenation.
\end{theorem}
\begin{proof}
By Lemma \ref{lem:id_is_quasitoric}, the identity element is in $qgB_n$ so that the set is non-empty. It is easy to verify that $qgB_n$ is closed under the concatenation. Now let $\beta \in qgB_n$. If $\beta$ is a pure braid, then by Theorem \ref{thm:pure_is_quasitoric}, $\beta^{-1}$ is in $qgB_n$. Now suppose that $\beta$ is not pure. Then note that for some non-negative integer $k$, $(\beta^{-1})^k$ is a pure braid, and by Theorem \ref{thm:pure_is_quasitoric}, $(\beta^{-1})^k$ is quasitoric. Observe that
\[ \beta^{-1}= (\beta^{-1})^k \beta^{k-1},\]
where both $(\beta^{-1})^k$ and $\beta^{k-1}$ are quasitoric, and thus $\beta^{-1}$ is quasitoric.
\end{proof}

\begin{ack}
NN has received funding from the European Union’s Horizon Europe Research and Innovation programme under the Marie Sklodowska Curie grant agreement no. 101066588. MS has received funding from Fulbright-Nehru Postdoctoral Fellowship grant 2865/FNPDR/2022.
\end{ack}

\bibliography{references1}{}
\bibliographystyle{alpha}
\end{document}